\documentclass[12pt,reqno]{amsart}
\usepackage{amssymb}
\usepackage{graphicx}
\usepackage[all]{xy}
\DeclareFontFamily{U}{mathb}{\hyphenchar\font45}
\DeclareFontShape{U}{mathb}{m}{n}{
      <5> <6> <7> <8> <9> <10> gen * mathb
      <10.95> mathb10 <12> <14.4> <17.28> <20.74> <24.88> mathb12
      }{}
\DeclareSymbolFont{mathb}{U}{mathb}{m}{n}
\DeclareMathSymbol{\righttoleftarrow}{3}{mathb}{"FD}

\theoremstyle{plain}
\newtheorem{prop}{Proposition}[section]
\newtheorem{theo}[prop]{Theorem}

\newtheorem{lemm}[prop]{Lemma}
\theoremstyle{remark}

\theoremstyle{definition}

\numberwithin{equation}{section}

\newcommand{\Q}{{\mathbb Q}}

\newcommand{\R}{{\mathbb R}}
\newcommand{\Z}{{\mathbb Z}}

\newcommand{\cB}{{\mathcal B}}

\newcommand{\cM}{{\mathcal M}}
\newcommand{\cS}{{\mathcal S}}

\newcommand{\ra}{\rightarrow}

\newcommand{\bQ}{{\mathbb Q}}
\newcommand{\bR}{{\mathbb R}}
\newcommand{\bZ}{{\mathbb Z}}

\newcommand{\eqto}{\stackrel{\lower1.5pt\hbox{$\scriptstyle\sim\,$}}\to}
\newcommand{\eqdashto}{\stackrel{\lower1.5pt\hbox{$\scriptstyle\sim\,$}}\dashrightarrow}

\newcommand{\actsfromright}{\righttoleftarrow}

\begin{document}
\title[Equivariant Burnside groups]{Arithmetic properties of equivariant birational types}

\author{Andrew Kresch}
\address{
  Institut f\"ur Mathematik,
  Universit\"at Z\"urich,
  Winterthurerstrasse 190,
  CH-8057 Z\"urich, Switzerland
}
\email{andrew.kresch@math.uzh.ch}
\author{Yuri Tschinkel}
\address{
  Courant Institute,
  251 Mercer Street,
  New York, NY 10012, USA
}

\email{tschinkel@cims.nyu.edu}

\address{Simons Foundation\\
160 Fifth Avenue\\
New York, NY 10010\\
USA}

\date{December 7, 2020}

\begin{abstract}
We study arithmetic properties of equivariant birational types introduced by Kontsevich, Pestun, and the second author. 

\end{abstract}

\maketitle

\section{Introduction}
\label{sec.intro}

Let $G$ be a finite abelian group and $k$ an algebraically closed field of characteristic zero. 
Investigations of obstructions to $G$-equivariant birationality over $k$
led to the definition, in \cite{kontsevichpestuntschinkel},
of new  
invariants of actions of $G$ on algebraic varieties $X$ defined over $k$.
These invariants were further developed in \cite{BnG},
where specialization maps were defined, generalizing the ones from
the non-equivariant setting \cite{KT}.

The invariants from \cite{kontsevichpestuntschinkel} are computed on a suitable 
smooth projective model $X$, where $G$ acts regularly. To such an action one associates a class 
\begin{equation}
\label{eqn:sum}
[X\actsfromright G] := \sum_{\alpha} \beta_{\alpha},
\end{equation}
where the sum is over components of the fixed point locus $F_{\alpha}\subset X^G$, and $\beta_{\alpha}$ are the characters of $G$
appearing in the tangent bundle to a point $x_{\alpha}\in F_{\alpha}$. Equivariant birational maps can be factored into sequences of blowups (and blowdowns) of smooth $G$-stable subvarieties, thanks to Equivariant Weak Factorization. To obtain an invariant, one imposes relations 
on the formal sums in \eqref{eqn:sum}, 
of the type
$$
[\tilde{X}\actsfromright G]  - [X\actsfromright G] =0,
$$
for every equivariant blowup $\tilde{X}\to X$.

This construction motivated the introduction of two closely related quotients of
the free abelian group $\cS_n(G)$, generated by symbols 
\begin{equation}
\label{eqn:sym}
\beta=[a_1,\ldots, a_n]=[a_{\sigma(1)}, \ldots, a_{\sigma(n)}], \quad \forall \sigma \in \mathfrak S_n, 
\end{equation}
where $\beta$ is an $n$-dimensional {\em faithful} representation of $G$ over $k$, i.e., a collection of characters $a_1,\ldots, a_n$ of $G$, up to permutation, spanning the character group of $G$. 

We have a diagram
\begin{equation}
\begin{split}
\label{eqn:dia}
\xymatrix{ 
\cS_n(G)  \ar[r]^{\mathsf b} &  \cB_n(G) \ar[d]^{\mu} \\
\cS_n(G) \ar[r]^{\mathsf m}  & \cM_n(G) 
}
\end{split}
\end{equation}
Here, the projection $\mathsf b$ is the quotient by the relation: 
\begin{itemize}
\item[($\mathrm{\bf B}$)]  {\bf Blow-up:} 
for all 
$
[a_1,a_2,b_1,\ldots,b_{n-2}]\in \cS_n(G)
$
one has 
\begin{align}
\begin{split}
\label{keyrelation}
[&a_1,a_2,b_1,\ldots,b_{n-2}]  =  \\
&\begin{cases}[a_1,a_2-a_1,b_1,\ldots,b_{n-2}] +  [a_1-a_2,a_2,b_1,\ldots,b_{n-2}],&a_1\ne a_2,\\
[0,a_1,b_1,\ldots,b_{n-2}],&a_1=a_2.\end{cases}
\end{split}
\end{align}
\end{itemize}
One of the main results in \cite{kontsevichpestuntschinkel} is the following

\begin{theo}
Let $X$ be a smooth projective algebraic variety of dimension $n$ over $k$, with a regular action of $G$. 
The class 
$$
[X\actsfromright G] \in \cB_n(G)
$$
is a well-defined $G$-equivariant birational invariant. 
\end{theo}

\

Numerical experiments revealed an interesting structure of the {\em other} quotient  of $\cS_n(G)$, 
namely, via the projection $\mathsf m$  in \eqref{eqn:dia}, which is defined as quotient by the relation: 

\begin{itemize}
\item[($\mathrm{\bf M}$)]  {\bf Modular blow-up:} 
for all 
$
[a_1,a_2,b_1,\ldots,b_{n-2}] \in  \cS_n(G)
$
one has 
\begin{align}
\begin{split}
\label{keyrelation-mod}
[&a_1,a_2,b_1,\ldots,b_{n-2}] = \\
&\,\,\,\,\,\, [ a_1,a_2-a_1,b_1,\ldots,b_{n-2}] +  [a_1-a_2,a_2,b_1,\ldots,b_{n-2}].
\end{split}
\end{align}
\end{itemize}

To distinguish, we write
$$
[a_1,\ldots, a_n], \quad \text{ respectively, }\quad \langle a_1,\ldots, a_n\rangle ,
$$ 
for the image of a generator in $\cB_n(G)$, respectively, 
the image of a generator in $\cM_n(G)$. 

When $a_1\neq a_2$, the relations are {\em identical};
the only difference is
\begin{align*}
[a_1,a_1, \ldots, a_n] & =   [a_1,0, \ldots, a_n]  \,\,\,\,\, \in \cB_n(G) \\ 
\langle a_1,a_1, \ldots, a_n\rangle & =    2\langle a_1,0,\ldots, a_n\rangle   \in \cM_n(G).
\end{align*}
The homomorphism in \eqref{eqn:dia}
\begin{equation}
\label{eqn:mu1}
\mu: \cB_n(G)\to \cM_n(G), \quad n\ge 2, 
\end{equation}
is defined on symbols by: 
$$
\mu([a_1,\ldots, a_n]) := \begin{cases} \langle a_1,\ldots, a_n\rangle & \text{ if all } a_i\neq 0,\\
 2\langle a_1,\ldots, a_n\rangle & \text{ if exactly one } a_i= 0, \\
 0 & \text{ otherwise}.
\end{cases}
$$
In \cite{kontsevichpestuntschinkel} it was shown that this map on symbols is compatible with relations.

\subsection*{Antisymmetry} 
\label{sect:anti2}

We have a diagram of homomorphisms 
\[
\xymatrix{ 
 \cB_n(G) \ar[d]_{\mu}  \ar@{>>}[r] & \cB^-_n(G) \ar[d]^{\mu^-}\\
 \cM_n(G)\ar@{>>}[r] & \cM_n^-(G)
}
\]
where the horizontal maps are projections 
to the corresponding quotients by the additional relation
\begin{equation*}
\label{eqn:anti}
[ -a_1,\ldots, a_n] =- [ a_1,\ldots, a_n],
\end{equation*}
defined only for nontrivial $G$. On symbols, the map $\mu^-$ is the same as $\mu$; 
its compatibility with defining relations is obvious.
 
\

In this note, we prove a comparison, left open in  \cite[Conjecture 8]{kontsevichpestuntschinkel}:

\begin{theo}
\label{theo:compare} 
Both homomorphisms $\mu$ and $\mu^-$ are isomorphisms, after tensoring with $\bQ$. 
\end{theo}

This implies that  
the main constructions connected with $\cM_n(G)$, from Sections 4,5,6, and 9 of 
\cite{kontsevichpestuntschinkel}, also apply to 
$\cB_n(G)\otimes \bQ$. We briefly sketch these structures:
\begin{itemize}
\item {\bf Lattices and cones:} elements $\langle a_1,\ldots, a_n\rangle\in \cM_n(G)$
can be  identified with isomorphism classes of triples
\begin{equation}
\label{eqn:chi}
(\mathbf{L}, \chi, \Lambda), 
\end{equation}
where $\mathbf L=\bZ^n$ is a lattice, $\chi\in \mathbf{L}\otimes A$, and $\Lambda\subset \mathbf{L}\otimes \bR$ is a 
basic simplicial cone.
Here $A$ denotes the character group of $G$, and by a basic simplicial cone
we mean one that is spanned by a basis of $\mathbf{L}$.
Concretely, choosing a basis $e_1,\ldots, e_n$ of lattice vectors spanning $\Lambda$, one can write
$$
\chi=\sum_{i=1}^n\, e_i\otimes a_i, \quad a_i\in A, 
$$ 
and put
$$
(\mathbf{L}, \chi, \Lambda) \mapsto \langle a_1,\ldots, a_n\rangle.
$$
Changing the basis spanning $\Lambda$ permutes the entries $a_1,\ldots, a_n$, and
relation $(\mathbf M)$ arises from decompositions of a simplicial cone into simplicial subcones.
We will discuss this in more detail in Section~\ref{sect:lattice}. 
\item {\bf Operations:} Given an exact sequence of groups
$$
0\ra G' \ra G\ra G''\ra 0
$$ 
there is a $\bZ$-bilinear {\em multiplication} homomorphism
$$
\nabla: \cM_{n'}(G')\otimes \cM_{n''}(G'') \ra \cM_{n'+n''}(G), \quad n',n''\ge 1, 
$$
which descends to the antisymmetric versions, as well as a {\em co-multiplication} homomorphism
$$
\Delta: \cM_{n'+n''}(G) \ra \cM_{n'}(G')\otimes \cM_{n''}^-(G''),
$$
(the minus on the second factor is not an error), 
which also comes with an antisymmetric version
$$
\Delta^-: \cM_{n'+n''}^-(G) \ra \cM_{n'}^-(G')\otimes \cM_{n''}^-(G''). 
$$
These homomorphisms allow to decompose $\cM_n(G)$ into {\em primitive} pieces, and reveal a rich internal structure. 
\item {\bf Hecke operators:} 
 The lattice-theoretic interpretation of $\cM_n(G)$ leads to the definition of commuting 
operators 
$$
T_{\ell,r}:  \cM_n(G) \otimes \bQ \to \cM_n(G)
$$
for all $1\le r\le n-1$ and primes $\ell$ not dividing the order of $G$. 
By Theorem~\ref{theo:compare}, the groups $\cB_n(G)\otimes \bQ$ also carry Hecke operators. 
\item {\bf Cohomology of arithmetic groups:} Let 
$$
\Gamma(G,n)\subset \mathrm{GL}_n(\bZ)
$$
be the stabilizer of $\chi$ in \eqref{eqn:chi}. Let 
\begin{itemize}
\item $\mathcal F_n$  be the $\bQ$-vector space generated by characteristic functions
of convex finitely generated rational polyhedral cones $\Lambda\subset \bR^n$, 
modulo those of dimension $\le n-1$,
\item $\mathrm{St}_n$ be the {\em Steinberg}-module, and 
\item $\mathrm{or}_n$ be the {\em sign of the determinant} module.
\end{itemize}
By \cite[Prop. 22]{kontsevichpestuntschinkel} and Theorem~\ref{theo:compare}, we
have a commutative diagram
\[
\xymatrix{
\mathcal M_n(G)\otimes \bQ  \ar@{>>}[r] \ar[d]_{\simeq} & \mathcal M_n^-(G)\otimes \bQ \ar[d]^{\simeq}  \\
H_0(\Gamma(G,n), \mathcal F_n)  \ar@{>>}[r] & H_0(\Gamma(G,n), \mathrm{St}_n\otimes \mathrm{or}_n)  \\
\mathcal B_n(G)\otimes \bQ  \ar@{>>}[r]  \ar[u]^{\simeq} & \mathcal B_n^-(G)\otimes \bQ \ar[u]_{\simeq},
}
\]
\end{itemize}

The connection between the groups $\cB_n(G)\otimes \bQ$, encoding invariants of abelian actions 
on algebraic varieties, and the theory of automorphic forms, via cohomology of congruence subgroups, 
seems intriguing to us. However, given the link between $\cB_n(G)$ and $\cM_n(G)$ it is natural
to seek a lattice theoretic interpretation of $\cB_n(G)$ as well. This is done in Section~\ref{sect:lattice}. One of the byproducts is
the definition of Hecke operators 
$$
T_{\ell,r} : \cB_n(G)\to \cB_n(G),
$$
where $\ell$ is a prime not dividing the order of $G$ and $1\le r\le n-1$, 
over the integers.

\medskip
\noindent
\textbf{Acknowledgments:}
We are very grateful to Brendan Hassett for his interest  and help on this and related projects. 
The first author was partially supported by the
Swiss National Science Foundation. 
The second author was partially supported by NSF grant 2000099.

\section{Comparison}
\label{sect:str}
We continue to assume that $G$ is a finite abelian group. 
This section is closely related to \cite[Sections 3, 5, and 11]{kontsevichpestuntschinkel}. 
In particular, we settle Conjecture 8 from {\em ibid}, asserting that
$$
\cB_n(G)\otimes \bQ\simeq \cM_n(G)\otimes \bQ.
$$ 

Our first result is a refinement of \cite[Prop.\ 3.2]{HKTsmall}.

\begin{theo}
\label{theo:pn}
Let $n\ge 2$. \\
$\mathrm{(i)}$ Let $p$ be a prime and $a\in (\bZ/p\bZ)^\times$.
The class
\[
[a,0, \ldots]+[-a,0, \ldots ]  \in \cB_n(\bZ/p\bZ)
\]
is zero when $p\le 5$, and is annihilated by $(p^2-1)/24$ when $p\ge 7$. \\
$\mathrm{(ii)}$ Let $N>1$ be an integer and $a\in (\bZ/N\bZ)^\times$.
Then
$$
[a,0, \ldots]+[-a,0, \ldots ]  \in \cB_n(\bZ/N\bZ)_{\rm tors},
$$
the subgroup of torsion elements.
\end{theo}

We start with a sequence of technical lemmas.

\begin{lemm}
\label{lemm:72}
For $a$, $b\in (\Z/p\Z)^\times$ we have
\[ [a,b]+[a,-b]=[a,0]  \in \cB_2(\Z/p\Z).
\]
\end{lemm}

\begin{proof}
We write $b=ma$ with $m\in \{1,\dots,p-1\}$ and
proceed by induction on $m$.
The base case $m=1$ is clear, since
\[ [a,a]=[a,0]\qquad\text{and}\qquad [a,-a]=0. \]
The induction hypothesis, in combination with
\[ [a,(m+1)a]=[a,ma]+[(m+1)a,-ma] \]
and
\[ [a,-ma]=[a,-(m+1)a]+[(m+1)a,-ma], \]
gives the inductive step.
\end{proof}

\begin{lemm}
\label{lemm:73}
For $a$, $b\in (\Z/p\Z)^\times$ we have
\[ [a,0]+[-a,0]=[a,b]+[a,-b]+[-a,b]+[-a,-b] \]
in $\cB_2(\Z/p\Z)$, and this element is independent of $a$ and $b$.
\end{lemm}

\begin{proof}
The equality holds by Lemma \ref{lemm:72}.
The right-hand side is symmetric in $a$ and $b$ and, by the equality,
is independent of $b$.
Hence it is also independent of $a$.
\end{proof}

\begin{lemm}
\label{lemm:74}
For $a$, $b\in (\Z/p\Z)^\times$ with $a+b\ne 0$, we have
\[ [a,0]=[a,b]+[-b,a+b]+[-a-b,a] \in \cB_2(\Z/p\Z).
\]
\end{lemm}

\begin{proof}
This follows from
\[ [a,-b]=[-b,a+b]+[-a-b,a] \]
and $[a,b]+[a,-b]=[a,0]$.
\end{proof}

Lemma \ref{lemm:73} tells us that
\begin{equation}
\label{eqn.delta}
\delta:=[a,0]+[-a,0]\in \cB_2(\Z/p\Z)
\end{equation}
is independent of $a\in (\Z/p\Z)^\times$.

\begin{lemm}
For $a$, $b\in (\Z/p\Z)^\times$ with $a+b\ne 0$, we have
\begin{align*}
\delta&=[a,b]+[-b,a+b]+[-a-b,a] \\
&\,\,\,\,\,\,\,\,+[-a,-b]+[b,-a-b]+[a+b,-a] \in \cB_2(\Z/p\Z).
\end{align*}
\end{lemm}

\begin{proof}
We add together
\[ [a,0]=[a,b]+[-b,a+b]+[-a-b,a] \]
and
\[ [-a,0]=[-a,-b]+[b,-a-b]+[a+b,-a], \]
and recognize $\delta$ on the left-hand side.
\end{proof}

\begin{lemm}
We have in $\cB_2(\Z/p\Z)$:
\begin{align*}
\sum_{a\in (\Z/p\Z)^\times} [a,a]&=\frac{p-1}{2}\delta, \\
\sum_{a\in (\Z/p\Z)^\times} [a,-2a]&=0.
\end{align*}
\end{lemm}

\begin{proof}
We pair summands indexed by $a$ and $-a$ and use $[a,a]=[a,0]$ and
the definition of $\delta$ to get the first equality.
Also, from
\[ [a,0]=[a,a]+[-a,2a]+[-2a,a] \]
follows the vanishing of pairs of summands in the second equality.
\end{proof}

\begin{lemm}
Let $\beta$, $\beta'$, $\beta''\in (\Z/p\Z)^\times\setminus \{-1\}$ with
\[ \beta'=-\beta^{-1}-1\qquad\text{and}\qquad \beta''=-(\beta+1)^{-1}. \]
Then
\[
\sum_{a\in (\Z/p\Z)^\times} [a,\beta a]+[a,\beta' a]+[a,\beta'' a]=\frac{p-1}{2}\delta.
\]
Furthermore, if $\beta=\beta'=\beta''$ then
\[
\sum_{a\in (\Z/p\Z)^\times} [a,\beta a]=\frac{p-1}{6}\delta.
\]
\end{lemm}

\begin{proof}
We identify pairs of summands in the first expression with $\delta$.
We have $\beta=\beta'=\beta''$ if and only if $\beta$ is a primitive cube root
of unity.
Then we may identify $6$-tuples of summands with $\delta$ to get the
second equality.
\end{proof}

\begin{proof}[Proof of Theorem \ref{theo:pn}]
For $\mathrm{(i)}$, let $p\ge 5$.
We partition $(\Z/p\Z)^\times\setminus \{1,-1\}$ into
$\{-2,-1/2\}$,
the primitive cube roots of unity (which exist only when $p\equiv 1$ mod $3$),
and $6$-element sets
\[ \{\beta,\beta',\beta'',\beta^{-1},\beta'^{-1},\beta''^{-1}\}, \]
with distinct $\beta$, $\beta'$, $\beta''$ as above.
We take a subset 
$$
I\subset (\Z/p\Z)^\times\setminus \{1,-1\},
$$ 
to
consist of one of $-2$, $-1/2$,
one primitive cube root of unity if it exists, and
$\beta$, $\beta'$, $\beta''$ from every $6$-element set as above.
Then $\delta$ from \eqref{eqn.delta} satisfies
\begin{align*}
\frac{(p-1)(p-2)}{6}\delta&=\sum_{\beta=1}^{p-3}\sum_{a\in (\Z/p\Z)^\times} [a,\beta a] \\
&=
\sum_{\beta\in I} \sum_{a\in (\Z/p\Z)^\times} [a,(\beta-1)a]+[a,(\beta^{-1}-1)a] \\
&=\sum_{\beta\in I} \sum_{a\in (\Z/p\Z)^\times} [a,\beta a] \\
&=\frac{(p-1)(p-5)}{12}\delta.
\end{align*}

It follows that $\delta$ is annihilated by $(p^2-1)/12$. 
Next we prove annihilation by $(p^2-1)/8$, and thus by $(p^2-1)/24$ as claimed.
This is an adaptation of the previous argument:
take 
$$
J\subset (\Z/p\Z)^\times \setminus  \{1,-1\}
$$ 
to consist of
one square root of $-1$ when $p\equiv 1$ mod $4$ as well as
$\beta$ and $-\beta$ from every $4$-element set
\[ \{\beta, -\beta, \beta^{-1}, -\beta^{-1}\}. \]
From
\begin{align*}
\frac{(p-1)^2}{4}\delta & =\sum_{\beta=1}^{p-3} \sum_{a\in (\Z/p\Z)^\times} [a, \beta a]  \\
 & =\sum_{\beta\in J}\sum_{a\in (\Z/p\Z)^\times} [a, \beta a]     =\frac{(p-1)(p-3)}{8}  \delta
\end{align*}
we get the desired conclusion.

For $\mathrm{(ii)}$, we treat composite $N$, as in the proof of \cite[Prop.\ 3.2]{HKTsmall}:
We recall that, for $a,b$ with $\gcd(a,b,N)=1$, we have
$$
\langle a,b\rangle= \begin{cases}  [a,b] & \text{ when both } a,b\neq 0, \\
\frac{1}{2} [a,0] & \text{ when } b=0.
\end{cases}
$$
In this case, we work with
$$
\delta(a,b) := \langle a,b\rangle +\langle -a,b\rangle+\langle a,-b\rangle + \langle -a,-b\rangle \in \cB_2(\bZ/N\bZ).
$$
We observe that $\delta(a,b)$ satisfies the blow-up relation ($\mathbf{M}$), thus
$$
S:=\sum_{a,b} \delta(a,b) = 2S.
$$
It follows that $S=0$ in $\cB_2(\bZ/N\bZ)$.
On the other hand, $\delta(a,b)$ 
is seen to be invariant under $\mathrm{SL}_2(\bZ/N\bZ)$. 
This implies that $\delta(a,b)$ is torsion in $\cB_2(\bZ/N\bZ)$ 
(annihilated by the number of summands in $S$). 
Substituting $b=0$, we
$[a,0]+[-a,0]=0$ in $\cB_2(\bZ/N\bZ)\otimes \bQ$.
\end{proof}

\

The following theorem settles Conjectures 8 and 9 of \cite{kontsevichpestuntschinkel}:

\begin{theo}
\label{prop:001N}
Let $n\ge 3$. \\
$\mathrm{(i)}$ Let $p$ be a prime. Then
\[
[0,0,1,\dots] \in \cB_n(\bZ/p\bZ)
\]
is zero when $p\le 5$, and is annihilated by $(p^2-1)/24$ when $p\ge 7$. \\
$\mathrm{(ii)}$ Let $G$ be a finite abelian group.
Any element of the form
\[
[0,0,\dots] \in \cB_n(G)
\]
is a torsion element.
\end{theo}

\begin{proof}
For $\mathrm{(ii)}$
it suffices to consider cyclic $G=\bZ/N\bZ$.
Theorem \ref{theo:pn} $\mathrm{(ii)}$ gives,
for $a\in (\bZ/N\bZ)^\times$, that
$$
[a,0,c \ldots] +[-a,0,c,\ldots]
$$
is torsion in $\cB_n(\bZ/N\bZ)$.
Substituting $c=a$, and using that
$$
[a,0,a,\dots]=[0,0,a,\dots] \in  \cB_n(\bZ/N\bZ)
$$
and
$$
[-a,0,a, \ldots ] =0\in  \cB_n(\bZ/N\bZ)
$$
we obtain the result.
We obtain $\mathrm{(i)}$ similarly, from Theorem \ref{theo:pn} $\mathrm{(i)}$.
\end{proof}

\begin{proof}[Proof of Theorem~\ref{theo:compare}]
The assertion for $\mu^-$ follows immediately from the
vanishing of all $[0,a_2,\ldots,a_n]$, respectively $\langle 0,a_2,\ldots,a_n\rangle$ 
in $\cB_n^-(G)\otimes \bQ$, respectively $\cM_n^-(G)\otimes \bQ$.

To obtain the assertion for $\mu$, we combine
Theorem \ref{prop:001N} $\mathrm{(ii)}$ with
analogous relations in $\cM_n(G)$, stated at the
beginning of Section 3 of \cite{kontsevichpestuntschinkel},
to show directly that $\mu$ induces an isomorphism after
tensoring with $\bQ$.
\end{proof}

\section{Interpretation via lattices}
\label{sect:lattice}
As before, $G$ is a finite abelian group $G$;
we denote by $A$ the character group of $G$.
Our starting point is the free abelian group on triples 
$$
(\mathbf{L}, \chi, \Lambda),
$$
where 
\begin{itemize}
\item $\mathbf{L}\simeq \bZ^n$ is an $n$-dimensional lattice,
\item  $\chi\in \mathbf{L}\otimes A $ is an element inducing, by duality, a surjection $\mathbf{L}^\vee \to  A$,
\item $\Lambda$ is a basic cone, i.e., a simplicial cone
spanned by a basis of $\mathbf{L}$.
\end{itemize}
Let $\mathbf{T}$ be the quotient of this group by the equivalence relation: two triples are equivalent if they differ by the action of 
$\mathrm{GL}_n(\bZ)$. 
There is a natural map
$$
\begin{array}{ccc}
 \mathbf{T}                     & \to           &  \cS_n(G), \\
           (\mathbf{L}, \chi, \Lambda) & \mapsto   & [a_1,\ldots, a_n],
\end{array}
$$
defined by decomposing 
\begin{equation}
\label{eqn.chi}
\chi=\sum_{i=1}^n e_i\otimes a_i, \quad a_i\in A, 
\end{equation}
where $\{e_1,\ldots, e_n\}$ is a basis of $\Lambda$. The symmetry property \eqref{eqn:sym} is precisely the ambiguity in the 
order of generating elements of $\Lambda$. 
Imposing scissor-type relations \cite[(4.4)]{kontsevichpestuntschinkel}  on $\mathbf{T}$, we obtain a diagram
\[
\xymatrix@R=8pt{
\mathbf{T} \ar[dr]^\psi \ar@{>>}[dd]_{\mathsf s}      \\
& \mathcal{M}_n(G) \\
\mathbf{T}/(\text{scissor-type relations})\ar[ur]^{\sim}
}
\]

We propose a similar group $\widetilde{\mathbf{T}}$, based on triples
$$
(\mathbf{L},\chi,\Lambda'),
$$ 
where now $\Lambda'$ is a
smooth cone of \emph{arbitrary} dimension (i.e., one spanned by part of
a basis of $\mathbf{L}$), such that when we let
$\mathbf{L}'$ denote the sublattice of $\mathbf{L}$ spanned by $\Lambda'$, we have
\begin{equation}
\label{eqn.chicondition}
\chi\in \mathrm{Im}(\mathbf{L}' \otimes A\to  \mathbf{L}\otimes A).
\end{equation}
Again, we impose the relations coming from the evident $\mathrm{GL}_n(\Z)$-action.
There is a natural map
$$
\begin{array}{ccl}
 \widetilde{\mathbf{T}}                     & \to       &  \cS_n(G), \\
           (\mathbf{L}, \chi, \Lambda') & \mapsto   &[a_1,\ldots, a_n].
\end{array}
$$

We introduce {\bf Subdivision relations} on $\widetilde{\mathbf{T}}$:
\begin{itemize}
\item[($\mathrm{\bf S}$)] 
for a face $\Lambda''$ of $\Lambda'$ of dimension at least $2$,
$$
\Lambda''=\R_{\ge 0}\langle e_1,\dots,e_r\rangle\subset 
\Lambda'=\R_{\ge 0}\langle e_1,\dots,e_s\rangle,
$$
consider the star subdivision $\Sigma^*_{\Lambda'}(\Lambda'')$,
consisting of the $2^r-1$ cones spanned by
$e_1+\dots+e_r$, $e_{r+1}$, $\dots$, $e_s$, and
all proper subsets of $\{e_1,\dots,e_r\}$.
Then 
\begin{align}
(\mathbf{L},\chi,\Lambda')&=
\sum_{\substack{\widetilde{\Lambda}'\in \Sigma^*_{\Lambda'}(\Lambda'')\\
\chi\in \mathrm{Im}(\widetilde{\mathbf{L}}' \otimes A\to  \mathbf{L}\otimes A)}} (-1)^{\dim(\Lambda')-\dim(\widetilde{\Lambda}')}(\mathbf{L},\chi,\widetilde{\Lambda}'),
\label{eqn.subdivision1} \\
(\mathbf{L},\chi,\Lambda')&=(\mathbf{L},\chi,\Lambda),
\label{eqn.subdivision2}
\end{align}
for a basic cone $\Lambda$, having $\Lambda'$ as a face.
\end{itemize}

We have:
\begin{equation}
\begin{split}
\label{eqn.stilde}
\xymatrix@R=8pt{
\widetilde{\mathbf{T}} \ar[dr]^{\tilde\psi} \ar@{>>}[dd]_{\tilde{\mathsf s}}      \\
& \cB_n(G) \\
\widetilde{\mathbf{T}}/(\text{subdivision relations})\ar[ur]^{\sim}
}
\end{split}
\end{equation}

\begin{lemm}
\label{lem.subdivisionr2}
The subdivision relations are generated by
\eqref{eqn.subdivision1} for $r=2$, and
\eqref{eqn.subdivision2}.
\end{lemm}

\begin{proof}
As in the proof of \cite[Prop.\ 2.1]{HKTsmall}, we show
inductively that the relations \eqref{eqn.subdivision1} for given $r>2$
are generated by \eqref{eqn.subdivision1} with smaller values of $r$.
\end{proof}

In \eqref{eqn.stilde} we have an obvious map from $\cB_n(G)$
to the quotient of $\widetilde{\mathbf{T}}$ by the subdivision relations,
sending $[a_1,\dots,a_n]$ to a triple $(\mathbf{L},\chi,\Lambda)$ with
$\Lambda$ a basic cone and $\chi$ given by the formula \eqref{eqn.chi}.
It is readily verified that this respects the relation \eqref{keyrelation},
and that the bottom map in \eqref{eqn.stilde} is an isomorphism.

As in \cite[Section~\ 4]{kontsevichpestuntschinkel} we extend the definition of
$\tilde\psi(\mathbf{L},\chi,\Lambda')$ to the case of a simplicial cone
$\Lambda'$, satisfying \eqref{eqn.chicondition} with
$\mathbf{L}'=\mathbf{L}\cap \Lambda'\otimes \R$.
We choose a subdivision by smooth cones and sum,
with signs, the contributions from the cones, not contained in any proper face
of $\Lambda'$.
Here, as in \eqref{eqn.subdivision1}, the signs are given by
codimension, and contributions are only taken from summands satisfying
the analogous condition to \eqref{eqn.chicondition}.

Now we can define Hecke operators
$$
T_{\ell,r} : \cB_n(G)\to \cB_n(G),
$$
where $\ell$ is a prime not dividing the order of $G$ and $1\le r\le n-1$,
following the construction in  \cite[Section 6]{kontsevichpestuntschinkel}, as a sum over certain
overlattices:
\[ T_{\ell,r}(\tilde\psi(\mathbf{L},\chi,\Lambda')):=
\sum_{\substack{\mathbf{L}\subset \widehat{\mathbf{L}}\subset
\mathbf{L}\otimes\Q \\ \widehat{\mathbf{L}}/\mathbf{L}\simeq (\Z/\ell\Z)^r}}
\tilde\psi(\widehat{\mathbf{L}},\chi,\Lambda'). \]

\bibliographystyle{plain}
\bibliography{number}

\end{document}